\documentclass[a4paper, english, 11pt]{report}
\usepackage{import}
\usepackage{amsmath}
\usepackage{amsthm}
\usepackage{mathrsfs}
\usepackage[english]{babel}
\usepackage[utf8]{inputenc}
\usepackage{extsizes}
\usepackage[T1]{fontenc}

\usepackage{fullpage} 
\usepackage{amsmath,amssymb,amsthm,amsfonts,amstext}
\usepackage{mathrsfs} 
\usepackage{mathtools}
\usepackage{thmtools}
\usepackage{alphalph}
\usepackage{pifont} 

\usepackage{latexsym}
\usepackage{xurl}
\usepackage{cancel} 
\usepackage{tikz,pgfplots} 
\usepackage{footnote}
\usepackage{color}
\usepackage{tikz-3dplot}
\usepackage[all]{xy}
\usepackage{titlesec} 
\usepackage{csquotes} 
\usepackage{multirow} 
\usepackage{nicefrac} 
\usepackage{empheq} 
\usepackage{physics} 

\usepackage{esvect} 

\usepackage{tensor} 

\usepackage[toc]{appendix}
\usepackage[font={it}]{caption}
\usepackage{subcaption}
\usepackage{xfrac}
\usepackage[backend=biber]{biblatex}
\usepackage[procnames]{listings}
\usepackage{enumitem} 
\usepackage{float} 

\usepackage{bm}

\usepackage{setspace}
\usepackage{graphicx}
\usepackage{subcaption} 
\usepackage{CJKutf8}
\usepackage[squaren, Gray, cdot]{SIunits}
\usepackage{framed}
\usepackage{transparent}
\usepackage{eso-pic}

\usepackage{color}

\usepackage{helvet}

\usepackage{mdframed}                               
\usepackage{wrapfig}

\definecolor{keywords}{RGB}{255,0,90}
\definecolor{comments}{RGB}{0,0,113}
\definecolor{red}{RGB}{160,0,0}
\definecolor{green}{RGB}{0,150,0}

\definecolor{shadecolor}{gray}{0.80}
\declaretheoremstyle[
headfont=\normalfont\bfseries,
notefont=\mdseries, notebraces={(}{)},
bodyfont=\normalfont,
postheadspace=0.5em,
spaceabove=0pt,
mdframed={
  skipabove=8pt,
  skipbelow=6pt,
  hidealllines=true,
  backgroundcolor={shadecolor},
  innerleftmargin=4pt,
  innerrightmargin=4pt}
]{shaded}

\theoremstyle{plain}
\newtheorem{thm}{Theorem}

\newtheorem{lemma}[thm]{Lemma}

\newtheorem{rmk}[thm]{Remark}

\theoremstyle{plain}
\newtheorem{defn}[thm]{Definition}

\theoremstyle{remark}

\usepackage[margin = 1.2in]{geometry}

\usepackage[utf8]{inputenc}
\usepackage{hyperref}

\usepackage{biblatex}
\usepackage{graphicx, color}
\graphicspath{{fig/}}

\usepackage{algorithm, algpseudocode} 
\usepackage{mathrsfs}

\author{Amirreza Harandi and Roger Moser}
\title{\LARGE The Sesquiharmonic Map Flow from Riemannian Surfaces}

\makeatletter
\def\@maketitle{%
  \newpage
  \null
  \vskip 2em%
  \begin{center}%
    {\Huge\bfseries\@title \par}%
    \vskip 1.5em%
    {\large
      \lineskip .5em%
      \begin{tabular}[t]{c}%
        \@author
      \end{tabular}\par}%
    \vskip 1em%
    {\large \@date}%
  \end{center}%
  \par
  \vskip 1.5em}
\makeatother

\numberwithin{equation}{section}

\begin{document}
\maketitle

\clearpage
\begin{center}
    \bfseries Abstract
\end{center}
Let $M$ be a two-dimensional compact manifold without boundary and let $N$ be a compact manifold without boundary. We study the $L^2$-gradient flow of an energy functional that interpolates between the harmonic map energy and the intrinsic biharmonic map energy. The critical points of this functional are called sesqui-harmonic maps. We investigate regularity properties of this flow, generalizing Struwe's classical regularity result for harmonic maps.

\section{Introduction}

For two smooth compact Riemannian manifolds $\left(M^m, g\right)$ and $\left(N^n, h\right)$ without boundary and a smooth map $u \in C^{\infty}(M, N)$, the Dirichlet energy $E_0(u)$ is defined by
\begin{equation}
 E_0(u):=\frac{1}{2} \int_M|D u|^2 d v_g =\int_M e(u)dv_g  
 \label{harmonic energy}
\end{equation}
where $e(u)$ denotes the Dirichlet energy density and $d v_g$ is the volume element of $(M, g)$.

Critical points of the Dirichlet energy are called harmonic maps. They satisfy the second-order elliptic system
\begin{equation}
\Delta u=0
\label{E-L harmonic}
\end{equation}
where $\Delta u$ is the tension field of $u$, also denoted by $\tau(u)$ in the literature. 

If the target manifold $N$ is isometrically embedded into a Euclidean space $\mathbb{R}^l$, then \eqref{E-L harmonic} is equivalent to
$$
\Delta u=\Delta^E u+A(u)(D u, D u)=0,
$$
where $\Delta^E$ denotes the Laplace--Beltrami operator of $(M, g)$ acting componentwise in $\mathbb{R}^l$ and $A$ is the second fundamental form of the embedding $N \hookrightarrow \mathbb{R}^l$.

The harmonic map heat flow is defined by
\begin{equation}
 \begin{aligned}
\partial_t u-\Delta^E u & =A(u)(D u, D u), \\
u(\cdot, 0) & =u_0.
\end{aligned}  
\label{hmf}
\end{equation}
The existence of solutions to the harmonic map heat flow was first studied by Eells and Sampson~\cite{eells1964harmonic}. They also proved that if the target manifold $N$ has nonpositive sectional curvature, then the harmonic map heat flow admits a global smooth solution.
In the case where the domain manifold $M$ is two-dimensional and the target manifold $N$ is arbitrary, Struwe \cite{struwe1985evolution} proved that for all $u_0 \in W^{1,2}(M, N)$ there exists a global weak solution of \eqref{hmf} which is smooth away from finitely many points.

A natural variant of harmonic maps is given by biharmonic maps. Intrinsic biharmonic maps are defined as critical points of the intrinsic bienergy
\begin{equation}
E_2(u):=\frac{1}{2} \int_M|\Delta u|^2 d v_g .
\label{biharmonic energy}    
\end{equation}
By \eqref{E-L harmonic}, harmonic maps minimize the intrinsic bienergy and are therefore intrinsic biharmonic maps. The Euler--Lagrange equation for $E_2$ was computed by Jiang \cite{jiang19862}:
\begin{equation}
 \Delta^2 u-R^N\left(D u\left(e_i\right), \Delta u\right) D u\left(e_i\right)=0,
\label{E-L biharmonic}   
\end{equation}
where $R^N$ denotes the curvature tensor of $N$ and $\left\{e_i\right\}$ is a local orthonormal frame on $M$. 

Here, if $\nabla^M$ denotes the Levi-Civita connection on $M$ and $\nabla$ denotes the induced connection on the pullback bundle $u^{-1}TN$, then the operator $\Delta^2 u = \Delta \tau$ denotes the Laplacian associated with the connection $\nabla$, defined by
\[
\Delta \tau = \nabla_{e_i}\nabla_{e_i}\tau - \nabla_{\nabla^M_{e_i} e_i}\tau.
\]

Jiang \cite{jiang19862} further showed that if the target manifold $N$ has non positive sectional curvature, then every intrinsic biharmonic map is harmonic.

Analogously, the intrinsic biharmonic heat flow is defined by
\begin{equation}
\begin{aligned}
\partial_t u & =-\Delta^2 u+R^N\left(D u\left(e_i\right), \Delta u\right) D u\left(e_i\right), \\
u(\cdot, 0) & =u_0.
\end{aligned}
\end{equation}
Ideas and arguments similar to those of Eells and Sampson~\cite{eells1964harmonic} and Struwe~\cite{struwe1985evolution} have also been applied to the biharmonic map heat flow~\cite{lamm2004heat, moser2005blowup, lamm2004biharmonic}, as well as to other geometric evolution equations, such as the polyharmonic map heat flow~\cite{gastel2006extrinsic}, the Yang--Mills flow~\cite{struwe1994asymptotic}, and the Willmore flow~\cite{kuwert2002gradient}.

For example in \cite{lamm2004biharmonic}, Lamm proved that if $\dim M \leq 4$ and $N$ has nonpositive sectional curvature, then the biharmonic map heat flow exists for all time and subconverges to a smooth harmonic map as $t \to \infty$.

Maps interpolating between harmonic and biharmonic maps are called \textit{sesqui-harmonic maps}. They are defined as critical points of the functional
\begin{equation}
    E_{ses}(u)=\delta_1\int_M |Du|^2dv_g+\delta_2\int_M|\Delta u|^2dv_g
\end{equation}
with $\delta_1,\delta_2\in [0,\infty)$, whose Euler-Lagrange equation is
\begin{equation}\label{elliptic sesqui harmoic map equ}
    \delta_1\Delta u+\delta_2(-\Delta^2u+R^N(Du(e_i),\Delta u)Du(e_i))=0.
\end{equation}
Similarly, the sesqui-harmonic map flow is defined by
\begin{equation}
    \begin{aligned}
    \partial _t u&=\delta_1\Delta u+\delta_2(-\Delta^2u+R^N(Du(e_i),\Delta u)Du(e_i)),\\
    u(\cdot,0)&=u_0.
    \end{aligned}\label{sesqui-harmonic map flow}
\end{equation}
Interpolating sesquiharmonic maps were studied by Branding~\cite{branding2020interpolating, branding2020some}. In contrast, the corresponding heat flow has received very little attention in the literature.

In this work we investigate the regularity of the flow \eqref{sesqui-harmonic map flow} 
for two dimensional domain manifolds $M$, extending Struwe’s theory for the harmonic map flow.
We restrict to nonnegative constants $\delta_1,\delta_2$ for several reasons. First, this assumption ensures that the flow \eqref{sesqui-harmonic map flow} is parabolic. Second, it guarantees the energy would be non increasing  and allows for control of the terms arising in the estimates used in the proof of our main results.

  We introduce a suitable notion of weak solution in section ~\ref{section 3.5} (definition ~\ref{definition of weak solution}) and prove global existence. 
Moreover, we show that the set of singularities is finite and that the solution 
is smooth away from these points. More precisely, we prove
\begin{thm}\label{Global solution}
    For any initial data $u_0 \in W^{2,2}(M;N)$, there exists a time $T=T(u_0)>0$ and a solution
\[
u \in C^\infty(M\times(0,T);N)\cap W^{1,2}(M\times [0,T);N)
\]
of \eqref{sesqui-harmonic map flow} with initial condition $u(\cdot,0)=u_0$. The maximal existence time $T(u_0)$ is characterized by the condition
\[
\limsup_{t\to T}\sup_{x\in M} E_0(u(t);B_R^M(x))>\epsilon_0
\qquad \text{for all } R>0,
\]
for some $\epsilon_0>0$. Moreover, the solution $u$ is smooth on $M\times(0,T]$ except at finitely many points $(x_k,T)$, $1\leq k\leq K$, characterized by
\[
\limsup_{t\to T} E_0(u(t);B_R^M(x_k))>\epsilon_0.
\]

Here, $B_R^M(x)$ denotes the geodesic ball in $M$ of radius $R$ centered at $x$, and we always assume that $R<i_M$, where $i_M$ denotes the injectivity radius of $M$.
\end{thm}
\begin{thm}\label{main theorem 2}
For any initial data $u_0\in W^{2,2}(M;N)$ there exists a weak solution $u$ of \eqref{sesqui-harmonic map flow} on $M\times[0,\infty)$ which is regular on $M\times[0,\infty)$ with exception of at most finitely many points.
\end{thm}
For the proof, refer to section~\ref{section 3.5}.

Finally, we prove in section~\ref{section blow up analysis} that at each singular point, 
the flow subconverges to a nontrivial harmonic bubble. More precisely we prove
\begin{thm}\label{scaling analysis theorem}
    Let $u$ be a solution to \eqref{sesqui-harmonic map flow}, constructed in Section~\ref{section 3.5}, and suppose that $(x_0,T)$ is a singular point satisfying
\[
\limsup_{t \to T} E_0(u(t); B_R^M(x_0)) > \epsilon_0, 
\quad \text{for all } R \in (0,R_0].
\]
Then there exist sequences $x_m \to x_0$, $t_m \to T$, and $R_m \to 0$ with $R_m \in (0,R_0]$, and a smooth harmonic map $v : \mathbb{R}^2 \to N$ such that, as $m \to \infty$,
\[
v_m(\cdot,0) \to v
\]
locally in $W^{2,2}(\mathbb{R}^2; N)$.

Moreover, the limit map $v$ has finite positive energy and extends to a nontrivial \textbf{harmonic map} $v : S^2 \to N$.
\end{thm}
For the proof, we refer to Section~\ref{section blow up analysis}.

The main differences between the present work and those of Lamm~\cite{lamm2004biharmonic} and Struwe~\cite{struwe1985evolution} stem from the structure of the flow itself. The equation \eqref{sesqui-harmonic map flow} combines both harmonic and biharmonic effects, and may be viewed as interpolating between the harmonic map heat flow and the biharmonic map heat flow. Moreover, the literature concerning this flow is rather limited.

A further difference arises in the construction of global solutions. In contrast to Lamm~\cite{lamm2004biharmonic}, where global smoothness follows from the assumption that the target manifold $N$ has nonpositive sectional curvature, no curvature assumption on $N$ is imposed in the present work. Consequently, we introduce a notion of global weak solution adapted to \eqref{sesqui-harmonic map flow}.

Since \eqref{sesqui-harmonic map flow} is highly nonlinear, with some nonlinearities hidden in the operator $\Delta$, defining weak solutions is not straightforward. Nevertheless, by carefully expanding the terms involved, we obtain a suitable notion of weak solution. This definition is compatible with the arguments of Struwe~\cite{struwe1985evolution}, while differing structurally from Struwe's notion of weak solutions.

In addition, since \eqref{sesqui-harmonic map flow} contains both harmonic and biharmonic terms, the scaling argument differs substantially from those used by Struwe~\cite{struwe1985evolution} and Lamm~\cite{lamm2004biharmonic}.

\section{Tools}\label{Sesuqi-Harmonic map flow}

 Here, we introduce some notation and definitions that we use in the following.

 For local estimates we introduce the local energies
$$E_2(u(t);B_R^M(x_0)):=\frac{1}{2}\int_{B_R^M(x_0)}|\Delta u(.,t)|^2dv_g,$$
$$E_0(u(t);B_R^M(x_0)):=\frac{1}{2}\int_{B_R^M(x_0)}|D u(.,t)|^2dv_g,$$
where $x_0\in M$ and $R>0$ as above and $B_R^M(x_0)\subset M$. Furthermore, we define a smooth cut-off function $\eta$ with the following properties
\begin{equation}\label{cut-off function}
\begin{aligned}
     &\eta\in C^\infty(M),\; \eta\geq0,\; \eta\equiv1\; \text{on}\; B_R^M(x_0),\\
     \; &\eta\equiv 0\; \text{on}\; M\setminus B_{2R}^M(x_0),||D\eta||_{L^\infty}\leq \frac{c}{R},||D^2\eta||_{L^\infty}\leq \frac{c}{R^2}
    \end{aligned}
\end{equation}
for $R<i_M$.

We use the letter $c$ as a generic constant that possibly depends on $M,N, u_0$, and other data.

We now state a lemma that will be used frequently in the sequel. The following result can be found in Ladyzhenskaya~\cite{ladyzhenskaia1968linear}, Lin and Wang~\cite{lin2008analysis}, and Struwe~\cite{struwe1985evolution}.

\begin{lemma}
For any $v \in W^{1,2}\left(\mathbb{R}^2\right)$, we have $v \in L^4\left(\mathbb{R}^2\right)$ and
\begin{equation}\label{G-N ineq}
\|v\|_{L^4\left(\mathbb{R}^2\right)}^4 \leq c\|v\|_{L^2\left(\mathbb{R}^2\right)}^2\|D v\|_{L^2\left(\mathbb{R}^2\right)}^2  
\end{equation}
for some constant $c$.
\end{lemma}

\section{Global and Local Estimates for Sesqui-harmonic map flow}
In this section we derive global and local evolution equations for the sesqui-harmonic energy, following the approach of \cite{struwe1985evolution}. 
\begin{lemma}\label{energy decreasing for sesqui-harmonic map}
    Let $u\in C^\infty(M\times [0,T),N)$ be a solution of \eqref{sesqui-harmonic map flow}. Then we have for all $t\in[0,T)$
    \begin{equation}
        E_{ses}(u(t))+\int_0^t\int_M |\partial_tu|^2dv_gdt=E_{ses}(u_0).
    \end{equation}
\end{lemma}
\begin{proof}
    For $E_2(u)$ we compute
    $$\delta_2\frac{d}{dt}E_2(u(t))=\int_M\langle\partial_tu(.,t),\delta_2(\Delta^2u-R^N(Du(e_i),\Delta u )Du(e_i))\rangle dv_g.$$
    We also compute
    $$\frac{d}{dt}E_0(u)=-\int_M \langle \partial _tu , \Delta u\rangle dv_g$$
    Replacing the second component in the inner product with flow, we obtain
    $$
    \begin{aligned}
         \delta_2\frac{d}{dt}E_2(u(t))&=\int_M\langle\partial_tu(.,t),-\partial_t u+\delta_1 \Delta u\rangle dv_g.\\
         &=-\int_M |\partial_tu|^2dv_g-\delta_1\frac{d}{dt}E_0(u).
     \end{aligned}
    $$
    Integrating over $[0,t]$ gives the desired result.
\end{proof}

\begin{lemma}\label{Key lemma}
    Let $u\in C^\infty (M\times[0,T),N)$ be a solution of \eqref{sesqui-harmonic map flow}. Then there are constants $c,R_0>0$ and $\epsilon_0>0$, such that for any $R\in(0,R_0]$ if $\sup\limits_{(x,t)\in M\times[0,T)}E_0(u(t);B_{2R}^M(x))<\epsilon_0$, then we have 
   \begin{equation}\label{L2-spacial bound for grad od du}
    \int_M|\nabla Du|^2\leq c(1+\frac{1}{R^2}).
    \end{equation}
    \begin{equation}\label{L4-spacial bound for du}
    \int_M|Du|^4\leq c(1+\frac{1}{R^2}).
    \end{equation}
    
    \begin{equation}\label{L2 bound for grad od du}
    \int_0^T\int_M|\nabla Du|^2\leq cT(1+\frac{1}{R^2}).
    \end{equation}
    \begin{equation}\label{L4 bound for du}
    \int_0^T\int_M|Du|^4\leq cT(1+\frac{1}{R^2}).
    \end{equation}

\end{lemma}
\begin{proof}
We choose $R_0$ as the global lower bound for $i_M$, the injectivity radius of $M$. In particular, $R_0$ is chosen such that $B_{2R_0}^M(x)\subset M$.
    According to Struwe~\cite{struwe1985evolution}, we have
   \begin{equation}\label{hessian L2 estimate}
    \begin{aligned}
     ||\nabla D u||_{L^2(M)}&\leq c||\Delta ^{E}u||_{L^2(M)}+c||Du||^2_{L^4(M)}+c||Du||_{L^2(M)}\\
     &\leq  c||\Delta u||_{L^2(M)}+||A(u)(Du,Du)||_{L^2(M)}+ c||Du||_{L^2(M)}+c||Du||^2_{L^4(M)}\\
     &\leq c||\Delta u||_{L^2(M)}+c||Du||^2_{L^4(M)}+ c||Du||_{L^2(M)}.
    \end{aligned} 
     \end{equation}
     Then let $\left\{\phi_i\right\} \subset C_0^{\infty}(M)$ be a partition of unity associated with a finite cover of $M$ by $B_{2 R}^M\left(x_i\right)$ with finite overlap, $0 \leq \phi_i \leq 1,\left|D \phi_i\right| \leq \frac{c}{R}$, for some constant $c$ and $\sum_i \phi_i^4=1$.

Arguing as in Lin and Wang \cite[Lemma 6.2.5]{lin2008analysis}, we have 
\begin{equation}
 \begin{aligned}
\int_M|D u|^4 & =\sum_i \int_M|D u|^4 \phi_i^4 \\
& \leq c \sup\limits_{(x,t)\in M\times[0,T]} E_0\left(u(t), B_{2 R}\left(x_i\right)\right)\left(\int_M\left|\nabla D u\right|^2+R^{-2} E_0\left(u_0\right)\right) \\
& \leq c \epsilon_0\left(\int_M\left|\nabla D u\right|^2+R^{-2} E_0\left(u_0\right)\right) .
\end{aligned}\label{6.23}   
\end{equation}

   Replacing \eqref{6.23} into \eqref{hessian L2 estimate} and using the fact that $||\Delta u||_{L^2(M)}$ and $||Du||_{L^2(M)}$ are bounded by initial energy and the energy is decreasing with respect to time, we get the result for \eqref{L2-spacial bound for grad od du} and then \eqref{6.23} gives \eqref{L4-spacial bound for du}. Integrating \eqref{L2-spacial bound for grad od du} and \eqref{L4-spacial bound for du} with respect to time gives \eqref{L2 bound for grad od du} and \eqref{L4 bound for du} respectively. 
\end{proof}

\begin{lemma}\label{L2-estimate for gradiant of laplacianof u}
  Let $u\in C^\infty (M\times[0,T),N)$ be a solution of \eqref{sesqui-harmonic map flow}. Then there exist constants $c_1$ and $c_2$, such that
  $$
  \int_0^t\int_M|\nabla\Delta u|^2dv_gdt\leq c_1+c_2\int_0^T\int_M|Du|^4.
  $$
\end{lemma}
\begin{proof}
    Multiplying \eqref{sesqui-harmonic map flow} by $\Delta u$ we get
    \begin{equation}\label{11}
        \langle\partial_t u , \Delta u\rangle=\delta_1\langle\Delta u, \Delta u\rangle-\delta_2\langle\Delta^2u,\Delta u\rangle+\delta_2\langle R^N(u)(Du,\Delta u)Du,\Delta u\rangle
    \end{equation}
   by integration by parts and divergence theorem 
   $$
   \begin{aligned}
    \int_M\langle\Delta^2u,\Delta u\rangle
    &=-\int_M|\nabla\Delta u|^2
   \end{aligned}
   $$
   and also 
   $$
   \int_0^T\int_M\langle\partial_t u,\Delta u\rangle=-\frac{1}{2}\int_0^T\frac{d}{dt}\int_M|Du|^2=E_0(u(0))-E_0(u(T))
   $$
   By replacing all these in \eqref{11} and using \eqref{G-N ineq}, Young's inequality and lemma~\ref{energy decreasing for sesqui-harmonic map} we get
   \begin{equation}
       \begin{aligned}
           C_0&\geq E_0(u(0))-E_0(u(T))=\delta_1\int_0^T\int_M|\Delta u|^2+\delta_2\int_0^T\int_M|\nabla\Delta u|^2\\
           &+\delta_2\int_0^T\int_M\langle R^N(u)(Du,\Delta u)Du,\Delta u\rangle\\
           &\geq \delta_2\int_0^T\int_M|\nabla\Delta u|^2-\epsilon\int_0^T\int_M|\Delta u|^4-c'\int_0^T\int_M|Du|^4\\
           &\geq \delta_2\int_0^T\int_M|\nabla\Delta u|^2-\epsilon\int_0^T(\int_M|\nabla\Delta u|^2)(\int_M|\Delta u|^2)-\int_0^T\int_M|Du|^4\\
           &\geq c\int_0^T\int_M|\nabla\Delta u|^2-c'\int_0^T\int_M|Du|^4
       \end{aligned}
   \end{equation}
   Where $c'$ depends on $\epsilon$ and $E_{ses}(u_0)$.

   as a result we get
   $$
   \int_0^T\int_M|\nabla\Delta u|^2\leq \frac{C_0}{c}+\frac{c'}{c}\int_0^T\int_M|Du|^4=c_1+c_2\int_0^T\int_M|Du|^4
   $$
\end{proof}

Now we derive a local monotonicity formula for Sesqui-Harmonic flow.

\begin{lemma}\label{monotonicity formula for sequi-harmonic maps}
Let $u\in C^\infty (M\times[0,T),N)$ be a solution of \eqref{sesqui-harmonic map flow}. Then there exists constants $R_0$ and $c=c(M,N)$, such that for any $R\in(0, R_0]$, there holds

\begin{equation}\label{monotonicity formula for harmonic part}
    E_0(u(T),B_R^M(X))\leq E_0(u(0),B_{2R}^M(x))+\frac{cT}{R^2}+cT.
\end{equation} 
\end{lemma}
\begin{proof}

Using the cut-off function $\eta$ defined in \eqref{cut-off function} and a local orthonormal frame $\{e_i\}$ We compute

\begin{equation}\label{proof of monotonicity formula}
    \begin{aligned}
\frac{1}{2}\frac{d}{dt}\int_M \eta^4 |Du|^2
&= -\int_M e_i(\eta^4)\langle \partial_t u, Du(e_i)\rangle - \int_M \eta^4 \langle \partial_t u, \Delta u\rangle \\
&\leq c \int_M \eta^4 |\partial_t u|^2 + \frac{c}{R^2} \int_M \eta^2 |Du|^2 + c \int_M \eta^4 |\Delta u|^2.
\end{aligned}
\end{equation}
Integrating in time and applying Lemma~\ref{energy decreasing for sesqui-harmonic map}, we obtain \eqref{monotonicity formula for harmonic part}.
\end{proof}

A series of estimates for biharmonic maps was established by Lamm~\cite{lamm2004biharmonic}. In the present setting, these arguments carry over with minor modifications. In~\cite{lamm2004biharmonic}, the $L^4$ norm of $Du$ is controlled using the assumption that the target manifold $N$ has nonpositive sectional curvature. Since no curvature assumption on $N$ is imposed in this paper, we instead rely on Lemma~\ref{Key lemma}.

Furthermore, the appearance of the factor $R^6$ in the denominator, instead of $R^4$ as in Lamm~\cite{lamm2004biharmonic}, is due to the additional factor $R^2$ arising in the denominator of Lemma~\ref{Key lemma}.

Apart from this modification, the remaining estimates follow by the same arguments as in Lamm~\cite{lamm2004biharmonic}. Therefore, we state these estimates without proof, except for \eqref{Laplace 2 of u}, whose proof uses \eqref{sesqui-harmonic map flow} directly.

\begin{lemma}\label{lamm lemma-1}
   Let $u\in C^\infty(M\times[0,T),N)$ be a solution of \eqref{sesqui-harmonic map flow} and let $\eta$ be as in \eqref{cut-off function}. Then there exist constants $c>0$, $R_0>0$ and $\epsilon_0>0$ such that if $\sup\limits_{(x,t)\in M\times [0,T)}E_0(u(t);B_{2R}^M(x))<\epsilon_0$, then for any $t\in [0,T)$ and $R\in(0,R_0]$ we have
   \begin{equation}\label{2.32}
       \int_0^t\int_M\eta^4|\nabla^2Du|^2\leq \frac{ct}{R^4}+c,
   \end{equation}
    \begin{equation}\label{counterpart of 2.32}
       \int_0^t\int_M\eta^4|D^3u|^2\leq \frac{ct}{R^4}+c.
   \end{equation}

    \begin{equation}\label{Laplace 2 of u}
       \int_0^t\int_M\eta^4|\Delta^2u|^2\leq \frac{ct}{R^6}+c.
   \end{equation}
   \begin{equation}\label{2.34}
       \int_0^t\int_M\eta^4|\nabla^3Du|^2\leq \frac{ct}{R^6}+c,
   \end{equation}
    \begin{equation}\label{counterpart of 2.34}
       \int_0^t\int_M\eta^4|D^4u|^2\leq \frac{ct}{R^6}+c.
   \end{equation}

\end{lemma}
\begin{proof}

  We use \eqref{sesqui-harmonic map flow} and replace $\Delta^2u$, with other terms in the equation. 
    $$
    \begin{aligned}
        \delta_2^2\int_0^t\int_M\eta^4|\Delta^2u|^2dv_g&=-\delta_2\int_0^t\int_M\eta^4\langle\partial_tu,\Delta^2 u\rangle+\delta_1\delta_2\int_0^t\int_M\eta^4\langle\Delta^2u,\Delta u\rangle\\
        &+\delta_2^2\int_0^t\int_M\langle R(Du,\Delta u)Du,\Delta^2u\rangle\\
         &=\int_0^t\int_M\eta^4|\partial_tu|^2-2\delta_1\int_0^t\int_M\eta^4\langle\partial_tu,\Delta u\rangle\\
        &-2\delta_2\int_0^t\int_M\eta^4\langle\partial_tu,R(Du,\Delta u)Du\rangle\\
        &+\delta_1^2\int_0^t\int_M\eta^4|\Delta u|^2+2\delta_1\delta_2\int_0^t\int_M\eta^4\langle\Delta u, R(Du,\Delta u)Du\rangle\\
        &+\delta_2^2\int_0^t\int_M\eta^4|R(Du,\Delta u)Du|^2\\
        &\leq c\int_0^t\int_M\eta^4|\partial_tu|^2+c\int_0^t\int_M\eta^4|\Delta u|^2\\
        &+c\int_0^t\int_M\eta^4|\partial_tu||Du|^2|\Delta u|+c\int_0^t\int_M\eta^4|\Delta u|^2|Du|^4\\
        &\leq c\int_0^t\int_M\eta^4|\partial_tu|^2+c\int_0^t\int_M\eta^4|\Delta u|^2\\
        &+c\int_0^t\int_M\eta^4|\Delta u|^4+c\int_0^t\int_M\eta^4|Du|^8.
    \end{aligned}
    $$

   Using \eqref{G-N ineq} and lemma~\ref{Key lemma}, we have
   \begin{equation}\label{L8 of Du}
        \begin{aligned}
       \int_M\eta^4|Du|^8=\int_M(\eta|Du|^2)^4&\leq c (\int_{[\eta>0]}|Du|^4)(\int(\eta^2|D\eta|^2|Du|^4+\eta^4 (D|Du|^2)^2))\\
       &\leq \frac{c}{R^6}+\frac{c}{R^2}\int_M\eta^4|Du|^2|\nabla Du|^2.
   \end{aligned}
  \end{equation} 
  For the last integral we use \eqref{2.32} and lemma~\ref{Key lemma} and obtain 
  \begin{equation}
  \begin{aligned}
    \int_M\eta^4|Du|^2|\nabla Du|^2&\leq(\int_{[\eta>0]}|Du|^4)^\frac{1}{2}(\int_M\eta^8|\nabla Du|^4)^\frac{1}{2}\\
     &\leq c\int_{[\eta>0]}|Du|^4\\
     &+c(\int_M\eta^4|\nabla Du|^2)(\int_M\eta^2|D\eta|^2|\nabla Du|^2+\int_M\eta^4|\nabla^2Du|^2).
     \end{aligned}
  \end{equation}
  As a result
  \begin{equation}\label{L8 of Du new result}
      \int_M\eta^4|Du|^8\leq \frac{c}{R^6}+c\int_M\eta^4|\nabla^2Du|^2.
  \end{equation}
   
   Once more, using \eqref{G-N ineq}  and  and bounded initial data gives
   \begin{equation}\label{L4 of laplacian new one}
    \begin{aligned}
       \int_M\eta^4|\Delta u|^4&\leq c(\int_M|\Delta u|^2)(\int_M\eta^2|D\eta|^2|\Delta u|^2+\int_M\eta^4|\nabla\Delta u|^2)\\
       &\leq \frac{c}{R^2}+c\int_\eta^4|\nabla\Delta u|^2.
   \end{aligned}
 \end{equation} 
Integrating \eqref{L8 of Du new result} and \eqref{L4 of laplacian new one} with respect to time, and using \eqref{2.32} together with Lemma~\ref{L2-estimate for gradiant of laplacianof u}, we obtain \eqref{Laplace 2 of u}.
\end{proof}

\subsection{Higher regularity}
In this section we follow the arguments of Lamm~\cite{lamm2004biharmonic} to derive $C^k$ estimates for solutions to \eqref{sesqui-harmonic map flow} on a short time interval prior to the first singular time, under a suitable smallness assumption on the energy. We begin by stating an $L^2$ estimate for $\partial_t u$.

\begin{lemma}\label{L2-estimate for time derivative}
    Let $u\in C^\infty(M\times[0,T),N)$ be a solution of \eqref{sesqui-harmonic map flow}. Then there exist constants $c$ and $\epsilon_0>0$, such that if $\sup\limits_{(x,t)\in M\times[0,T)}E_{0}(u(t);B_{2R}^M(x))<\epsilon_0$, there exists $0<\delta<min\{T,cR^6\}$, such that for all $s,t\in[0,T)$ with $s<t$ and $|t-s|<\delta$ we have
    \begin{equation}
        \sup\limits_{s\leq t'\leq t}\int_M|\partial_tu(t')|^2\leq c(R)\int_M|\partial_tu(s)|^2+c(R).
    \end{equation}
\end{lemma}
\begin{proof}
The proof is the same as that of Lamm~\cite{lamm2004biharmonic}.

\end{proof}
Using the previously derived estimates, we establish $L^2$ estimates for higher-order derivatives of $u$, in particular for $D^4 u$.
\begin{lemma}\label{lamm lemma 2}

    Let $u\in C^\infty(M\times[0,T),N)$ be a solution of \eqref{sesqui-harmonic map flow}. Then there exist constants $c$ and $\epsilon_0>0$, such that if $\sup\limits_{(x,t)\in M\times[0,T)}E_0(u(t);B_{2R}^M(x))<\epsilon_0$, there exists $0<\delta<min\{T,cR^6\}$, such that for all $0<\frac{\delta}{4}\leq t'<T$ we have
    \begin{equation}\label{nabla3du L2(M)-estimate}
        \int_M|\nabla^3Du(.,t')|^2\leq c,
    \end{equation}
    \begin{equation}\label{D4u L2(M)-estimate}
        \int_M|D^4u(.,t')|^2\leq c.
    \end{equation}
    The constant $c$ depends on $E_0(u_0)$, $E_2(u_0)$, $M$, $N$, $R$. 
\end{lemma}
\begin{proof}
The proof is essentially the same as in Lamm~\cite{lamm2004biharmonic}, since the additional term arising from the harmonic part, namely $\|\Delta u\|_{L^2(M)}$, is bounded in terms of the initial data.

\end{proof}
We use \eqref{nabla3du L2(M)-estimate} and \eqref{D4u L2(M)-estimate} to prove the higher regularity.
\begin{thm}\label{uniform holder bounds for all derivatives}
     Let $u\in C^\infty(M\times[0,T),N)$ be a solution of \eqref{sesqui-harmonic map flow}. Then there exists $\epsilon_0>0$, such that if $\sup\limits_{(x,t)\in M\times[0,T)}E_0(u(t);B_{2R}^M(x))<\epsilon_0$, there exists $0<\delta<min\{T,cR^6\}$ such that the Holder norms of $u$ and all derivatives of $u$ are uniformly bounded on $[t-\frac{\delta}{4},t]$ for all $t\in [\frac{\delta}{2},T)$, by constants which depend only on $E_0(u_0)$, $E_2(u_0)$, $M$, $N$, $R$ and the order of derivatives of $u$.
\end{thm}
\begin{proof}
The proof is the same as that of Lamm~\cite{lamm2004biharmonic}.

\end{proof}

\begin{rmk}\label{estimates at first singular time}
    All the estimates in Lemmas~\ref{L2-estimate for time derivative}, \ref{lamm lemma 2}, and Theorem~\ref{uniform holder bounds for all derivatives} can also be formulated locally by introducing a cutoff function $\eta$ as in \eqref{cut-off function} and deriving the corresponding local estimates. This is precisely the approach taken in Lamm~\cite{lamm2004biharmonic}. The global estimates stated above are then obtained by combining these local estimates with a covering lemma due to Struwe~\cite{struwe1985evolution} and a partition of unity argument.

    If the initial data $u_0$ are sufficiently regular, then the corresponding estimates remain valid on the entire interval $[0,T)$. We omit the proofs of these generalizations, since they rely on essentially the same arguments.
\end{rmk}

\section{Global existence for the sesqui-harmonic map flow}\label{section 3.5}

In the following, we introduce a suitable notion of weak solution of \eqref{sesqui-harmonic map flow}. First, We expand $\Delta^2u$ (at a single point).
$$
\begin{aligned}
\Delta^2u&=\nabla_{e_i}\nabla_{e_i}\Delta u\\
&=\nabla_{e_i}(D_{e_i}\Delta u-D^2\pi_N(u)(Du(e_i),\Delta u))\\
&=D_{e_i}(D_{e_i}\Delta u-D^2\pi_N(u)(Du(e_i),\Delta u))\\&-D^2\pi_N(u)(Du(e_i),D_{e_i}\Delta u
-D^2\pi_N(Du(e_i),\Delta u))\\
&=\Delta^E\Delta u-D^3\pi_N(u)(Du(e_i),Du(e_i),\Delta u)\\
&-D^2\pi_N(u)(\Delta^Eu,\Delta u)-2D^2\pi_N(u)(Du(e_i),D_{e_i}\Delta u)\\
&+D^2\pi_N(u)(Du(e_i),D^2\pi_N(Du(e_i),\Delta u))\\
&=\Delta^E\Delta u-2\operatorname{div}(D^2\pi_N(u)(Du,\Delta u))\\
&+D^2\pi_N(u)(\Delta^Eu,\Delta u)+D^3\pi_N(Du(e_i),Du(e_i),\Delta u)\\
&+D^2\pi_N(u)(Du(e_i),D^2\pi_N(Du(e_i),\Delta u)),
\end{aligned}
$$
Where we used the fact that,
$$
\begin{aligned}
\operatorname{div}(D^2\pi_N(u)(Du,\Delta u))&=D_{e_i}(D^2\pi_N(u)(Du(e_i),\Delta u))  \\
&=D^2\pi_N(u)(Du(e_i),D_{e_i}\Delta u)+D^2\pi_N(u)(\Delta^Eu,\Delta u)\\
&+D^3\pi_N(u)(Du(e_i),Du(e_i),\Delta u).
\end{aligned}
$$
By replacing the above for $\Delta^2u$ in \eqref{sesqui-harmonic map flow}, we obtain
$$
\begin{aligned}
    \partial_tu&=\delta_1\Delta u+\delta_2(-\Delta^E\Delta u-2\operatorname{div}(D^2\pi_N(u)(Du(e_i),\Delta u))+D^2\pi_N(u)(\Delta^Eu,\Delta u)\\
    &+D^3\pi_N(Du(e_i),Du(e_i),\Delta u)+D^2\pi_N(u)(Du(e_i),D^2\pi_N(Du(e_i),\Delta u))\\
    &+R^N(u)(Du(e_i),\Delta u)Du(e_i)).
\end{aligned}
$$
We rewrite the curvature term by the Gauss equation; which states that for $X,Y,Z$, the sections of the tangent bundle,we have
$$
S(X,A(Y,Z))-S(Y,A(X,Z))=R(X,Y)Z.
$$
So we can derive
$$
\begin{aligned}
    D^2\pi_N(Du(e_i),D^2\pi_N(\Delta u,Du(e_i))&-D^2\pi_N(\Delta u,D^2\pi_N(Du(e_i),Du(e_i))\\
    &=R^N(u)(Du(e_i),\Delta u)Du(e_i).
\end{aligned}
$$
And of course for the harmonic part we have
$$
\Delta u=\Delta^Eu-D^2\pi_N(u)(Du(e_i),Du(e_i)).
$$
Now that \eqref{sesqui-harmonic map flow} is rewritten as a map from $M$ into $\mathbb{R}^l$, It makes sense to define the test function $\Phi\in C^\infty_0(M\times[0,T];\mathbb{R}^l)$ and use the Euclidean inner product. Multiplying \eqref{sesqui-harmonic map flow} with $\Phi$ and integrating with respect to the spacial variable and integrating by part we obtain
\begin{equation}\label{spacial weak solution}
    \begin{aligned}
        \int_M\langle\partial_tu,\Phi\rangle dv_g&=\delta_1\int_M\langle\Delta^Eu,\Phi\rangle dv_g\\
        &+\delta_1\int_M\langle D^2\pi_N(u)(Du(e_i),Du(e_i)),\Phi\rangle dvg\\
        &-\delta_2\int_M\langle\Delta u,\Delta^E\Phi\rangle+\delta_2\int_M\langle-2\operatorname{div}(D^2\pi_N(u)(Du(e_i),\Delta u))\\
        &+D^2\pi_N(u)(\Delta^Eu,\Delta u)+D^3\pi_N(Du(e_i),Du(e_i),\Delta u)\\
        &+2D^2\pi_N(u)(Du(e_i),\Delta u)\\
        &-D^2\pi_N(\Delta u,D^2\pi_N(Du(e_i),Du(e_i))),\Phi\rangle dv_g.
    \end{aligned}
\end{equation}
 
\begin{defn}\label{definition of weak solution}
    $u$ is a weak solution of \eqref{sesqui-harmonic map flow} if $u\in W^{1,2}(M\times[0,T];N)$ s.t. for a.e. $t\in[0,T]$, $u(x,.)\in W^{2,2}(M;N)$ and solves \eqref{spacial weak solution} in the weak form.
\end{defn}

Now we prove theorem~\ref{Global solution}.
\begin{proof}
It was shown by Moser~\cite{moser2015L} and Zhu~\cite{zhu2014bubble} that if $u\in W^{1,2}(M;N)$ and $\Delta u\in L^2(M;\mathbb{R}^l)$, then $u\in W^{2,2}(M;N)$. 

Since $C^\infty(M;N)$ is dense in $W^{2,2}(M;N)$, we choose a sequence of smooth initial data $u_0^m$ converging to $u_0$ in $W^{2,2}(M;N)$. Let $u^m$ denote the corresponding local smooth solutions of \eqref{sesqui-harmonic map flow} with initial data $u_0^m$ for which all the estimates derived in the previous section are valid. Since $u^m_0\rightarrow u_0$ in $W^{2,2}(M;N)$ there exists $R>0$ such that $E_0(u^m_0;B_{2R}(x))\leq \frac{\epsilon_0}{2}$ for all $x\in M$ and all $m$.
By \eqref{monotonicity formula for harmonic part}, we have
$$
E_0(u^m(T),B_R^M(X))\leq \epsilon_0
$$
holds up to some time $T$ of order $\epsilon_0R^2$. By Theorem~\ref{uniform holder bounds for all derivatives}, there exists $\delta>0$ such that for every $t\in[\frac{\delta}{2},T)$, the interval $[t-\frac{\delta}{4},t]$ is contained in $(0,T)$, and the Hölder norms of $u^m$ and all its derivatives are bounded on this interval by a constant $c(\delta)$ depending  on the order of the derivative.

Let $K\subset(0,T)$ be a compact subset. Covering $K$ by finitely many overlapping intervals of this form and taking the maximum of the corresponding constants, we obtain uniform bounds on the Hölder norms of $u^m$ and all its derivatives on $K$.

Therefore, by the Arzelà--Ascoli theorem, a subsequence of $u^m$ converges to a smooth solution $u$ on every compact subset of $M\times(0,T)$. 

Moreover, since $\partial_t u^m$ is uniformly bounded in $L^2(M\times[0,T];N)$ and $u_0^m \to u_0$ in $W^{2,2}(M;N)$, the fundamental theorem of calculus yields
$$
u(t) \to u_0 \quad \text{in } L^2(M;N) \quad \text{as } t \to 0.
$$    
As a result $u$ is a solution of \eqref{sesqui-harmonic map flow} with initial data $u_0$. 

To characterize the maximal time of existence of a smooth solution, assume that for some $R$ we have
\[
\limsup_{t\to T} E_0(u(t);B_R^M(x)) \leq \epsilon_0.
\]
By Theorem~\ref{uniform holder bounds for all derivatives}, the Hölder norms of $u$ and all its derivatives remain bounded on the interval $[T-\frac{\delta}{4},T]$. In particular, $u(\cdot,T)$ is smooth. Hence, $u$ can be extended smoothly beyond $T$, which contradicts the maximality of $T$.

The finiteness of the singular set follows from the additivity of the energy, lemma~\ref{energy decreasing for sesqui-harmonic map}, and \eqref{monotonicity formula for harmonic part}. Let $\{x_j\}_{j=1}^K$ be any finite subset of singular points at time $T$. then we have
\[
\limsup_{t\to T} E_0(u(t);B_R^M(x_j)) > \epsilon_0.
\]
for any $R>0$ and $1\leq j\leq K$. If we choose $R>0$ small enough such that $B_{2R}^M(x_j),\; 1\leq j\leq K$ are mutually disjoint, then by \eqref{monotonicity formula for harmonic part}, we have
$$\begin{aligned}
    K\epsilon_0&\leq \sum\limits_{1\leq j\leq K}\limsup_{t\to T} E_0(u(t);B_R^M(x_j))\\
    &\leq \sum\limits_{1\leq j\leq K}((E_0(u(\tau),B_{2R}^M(x_j))+\frac{\epsilon_0}{2})\\
    &\leq E_0(u(\tau))+\frac{K\epsilon_0}{2}
\end{aligned}
$$
for any $\tau\in [T-\frac{\epsilon_0R^2}{2c(1+R^2)},T]$, where $c$ is the constant in \eqref{monotonicity formula for harmonic part}. Therefore $K\leq \frac{2E(u_0)}{\epsilon_0}$ and the set of singular points at time $T$ is finite.
\end{proof}
Now we prove theorem~\ref{main theorem 2}.

\begin{proof}

As in the proof of Theorem~\ref{Global solution}, the fundamental theorem of calculus implies that there exists $u(T)$ such that
\[
u(t) \to u(T) \qquad \text{in } L^2(M).
\]
Since $\|Du(t)\|_{L^2(M)}$ is uniformly bounded, we can extract a subsequence of the above sequence that converges weakly to $u(T)$ in $W^{1,2}(M;N)$. Together with the fact that $\Delta u(T) \in L^2(M;\mathbb{R}^l)$, this implies that $u(T) \in W^{2,2}(M;N)$. For further details, we refer to Moser~\cite{moser2015L} and Zhu~\cite{zhu2014bubble}.
Moreover, due to remark~\ref{estimates at first singular time}, $u(T)$ is actually smooth apart from singular points on every compact subset of $M\times \{T\}\setminus\{x_k\}_{k=1}^K\times\{T\}$.

Now assume $v$ be the solution on $(T,T_2)$ with initial data $u(T)$. Define $w$ on $(0,T_2)$ by
\[
w(t)=
\begin{cases}
v(t), & T<t<T_2,\\
v(T)=u(T), & t=T,\\
u(t), & 0<t<T,\\
u_0=u(0), & t=0.
\end{cases}
\]
Then $w$ satisfies \eqref{spacial weak solution} for a.e. $t\in(0,T_2)$ and is a weak solution of \eqref{sesqui-harmonic map flow}.

Since the initial energy is finite and nonincreasing by Lemma~\ref{energy decreasing for sesqui-harmonic map}, and since each singularity carries at least $\epsilon_0$ of energy, the total number of singularities is finite. For a detailed proof, we refer to Struwe~\cite{struwe1985evolution}.

As a result, iterating the above construction beyond $T_2$ and using that the solution becomes smooth after the final singular time $T_n$, we obtain a global weak solution defined for all $t\in(0,\infty)$.
\end{proof}

\section{Blow up  analysis}\label{section blow up analysis}
In the following, we analyze the behavior of solutions to \eqref{sesqui-harmonic map flow} in a neighborhood of a singular point. Let $(x_0,T) \in M \times \mathbb{R}$ be a singular point, and let $u$ be a solution of \eqref{sesqui-harmonic map flow}.
We recall that there exists $\epsilon_0 > 0$ such that if
\[
\sup_{(x,t)\in M\times[0,T)} E_0(u(t); B_{2R}^M(x)) \leq \epsilon_0
\]
for some $R > 0$, then all previously derived estimates, including \eqref{D4u L2(M)-estimate}, remain valid. Consequently, singularities can only occur through concentration of the harmonic energy.

This observation motivates rescaling the solution $u$ around the singular point in order to capture the behavior of the flow near points of energy concentration. 

Now we prove theorem~\ref{scaling analysis theorem}.

\begin{proof}

As $(x_0,T)$ is a singular point, by theorem~\ref{uniform holder bounds for all derivatives},  for any $t_m<T$ there exist a time $T_1<t_m$ such that the solution $u$ is smooth on $[T_1,t_m]$. Following Struwe~\cite{struwe1985evolution} and Lamm~\cite{lamm2004biharmonic}, there exist sequences $x_m \to x_0$, $t_m \to T$ with $t_m < T$, and $R_m \to 0$ with $R_m \in (0,R_0]$, such that

\begin{equation}
\begin{aligned}
    &\int_{B_{R_m}^M(x_m)}|Du(t_m)|^2dv_g=\sup\limits_{(x,t)\in B_{2R_0}(x_0)\times[t_m-\tau_m,t_m]}\int_{B_{R_m}^M(x)}|Du(t)|^2dv_g=\frac{\epsilon_0}{4},
    \end{aligned}
\end{equation}
Where $\tau_m=\frac{\epsilon_0R^2_m}{8c}$ and $c$ is the constant in \eqref{monotonicity formula for harmonic part}. 
We define the scaled function $v_m$ by
$$
v_m(x',t') := u(x_m + R_m x',\, t_m + R_m^2 t').
$$
Replacing $u$ by the scaled function $v_m$ in the above sequence, we obtain

\begin{equation}
\begin{aligned}
    \int_{B_1(0)}|Dv_m(t')|^2dv_{g_m}=\sup\limits_{(x',t')\in B_{\frac{2R}{R_m}}(0)\times[-\frac{\epsilon_0}{8c},0]}\int_{B_1(x')}|Dv_m(t')|^2dv_{g_m}=\frac{\epsilon_0}{4}.
    \end{aligned}
\end{equation}
By the fact that $t_m-\tau_m\rightarrow T$ and lemma~\ref{energy decreasing for sesqui-harmonic map}, we can see that as $m\rightarrow\infty$,
\begin{equation}\label{scaled L2 for partialt}
    \int_{-\frac{\epsilon_0}{8c}}^0\int_{B_{\frac{R}{R_m}}(0)}|\partial_tv_m|^2\leq\int_{t_m-\tau_m}^{t_m}\int_M|\partial_tu|^2\rightarrow0,
\end{equation}
where $g_m$ is the Riemannian metric in the rescaled variables.

Also for large $m$, by \eqref{2.34},  we obtain
\begin{equation}\label{scaled version L2 of 4th derivative}
    \int_{-\frac{\epsilon_0}{8c}}^0\int_{B_1(x')}|D^4v_m|^2\leq R_m^4\int_{t_m-\tau_m}^{t_m}\int_M|D^4u|^2\leq R^4_m\times\frac{c\tau_m}{R_m^6}\leq c,
\end{equation}
For all $x'\in B_{\frac{R}{R_m}}(0)$ and the constant $c$ in \eqref{2.34}.

As in Lamm~\cite{lamm2004biharmonic}, we can choose a sequence $s_m \in \left[-\frac{\epsilon_0}{8c},0\right]$ such that, for the sequence
$$
\bar{v}_m(x'):=v_m(x',s_m)=u(x_m+R_mx',t_m+R_m^2s_m)
$$
we have
\begin{equation}\label{scaled time deriavative sequence}
\partial_tv_m(.,s_m)\rightarrow0\;\;\; \text{in}\;L^2(\mathbb{R}^2;N)
\end{equation}
and 
\begin{equation}\label{scaled spacial derivative sequence}
    \bar{v}_m\rightarrow \bar{v}\;\;\;\text{weakly in}\; W_{loc}^{4,2}(\mathbb{R}^2;N)\;\text{strongly in }\;W_{loc}^{2,2}(\mathbb{R}^2;N).
\end{equation}

Denote the biharmonic part of \eqref{sesqui-harmonic map flow} by $f$. Using \eqref{sesqui-harmonic map flow}, we may write
\[
f := \Delta^2 u + R(u)(\Delta u,Du)Du
= \frac{\delta_1}{\delta_2}\Delta u - \frac{1}{\delta_2}\partial_t u.
\]
Hence, by \eqref{sesqui-harmonic map flow}  Lemma~\ref{energy decreasing for sesqui-harmonic map}, we obtain
\[
\begin{aligned}
\lim_{m\to \infty}\int_{t_m-\tau_m}^{t_m}\int_M |f|^2
&\leq \lim_{m\to \infty} c\left(
\int_{t_m-\tau_m}^{t_m}\int_M |\Delta u|^2
+
\int_{t_m-\tau_m}^{t_m}\int_M |\partial_t u|^2
\right)
= 0.
\end{aligned}
\]

We now rescale the biharmonic part by setting
\[
g_m(x',t') := R_m^2 f(x_m + R_m x',\, t_m + R_m^2 t').
\]
Then we obtain
$$
\int_{B_{\frac{R_0}{R_m}}}|g_m|^2=R_m^2\int_{B_{\frac{R_0}{R_m}}}(\frac{\delta_1}{\delta_2})^2|\Delta v_m|^2+R_m^2\int_{B_{\frac{R_0}{R_m}}}(\frac{1}{\delta_2})^2|\partial_{t'}v_m|^2-R_m^2\int_{B_{\frac{R_0}{R_m}}}(\frac{2\delta_1}{\delta_2^2})\langle\Delta v_m,\partial_{t'}v_m\rangle
$$
Using the same sequence constructed above together with \eqref{scaled time deriavative sequence} and \eqref{scaled spacial derivative sequence}, we conclude that
$
g_m \to 0
$
in $L^2_{\mathrm{loc}}(\mathbb{R}^2)$
as $m\to\infty$.

After rescaling \eqref{sesqui-harmonic map flow}, we obtain
\[
\partial_{t'} v_m = \delta_1 \Delta v_m - \delta_2 g_m.
\]
Since $g_m \to 0$, passing to the limit yields
\[
\Delta \bar{v}=0,
\]

in the weak sense, with positive energy. In fact Replacing $u$ by the rescaled function $v_m$ and $R$ by $R_m$ in \eqref{monotonicity formula for harmonic part}, we obtain, for $s = s_m + R_m^2 s' < t = t_m + R_m^2 t'$,

\begin{equation}\label{monotonicity formula for the scaled function}
    \begin{aligned}
        \int_{B_1(0)}|Dv_m(t')|^2dv_{g_m}&\leq \int_{B_{2}(0)}|Dv_m(s')|^2dv_{g_m}+{c(1+R^2_m)(t'-s')}.
    \end{aligned}    
\end{equation}

Now if we choose $t'-s'<\frac{\epsilon_0}{8c}$, we obtain
$$
E(\bar{v};B_2(0))=\lim\limits_{m\rightarrow\infty}\int_{B_2(0)}|Dv_m(s')|^2dv_{g_m}\geq \frac{\epsilon_0}{8}.
$$
Smoothness of stationary harmonic maps in two dimensions was established by Hélein~\cite{helein1991regularity}.
\end{proof}

\section{Acknowledgment}
This work was supported by the Engineering and Physical Sciences Research Council [grant number EP/W524712/1].

\printbibliography

\end{document}